\documentclass[11pt]{article} 

\usepackage[utf8]{inputenc} 

\usepackage{geometry} 
\usepackage{graphicx} 
\usepackage{amsmath} 
\usepackage{amsfonts} 
\usepackage{amsthm} 
\usepackage{amssymb} 
\usepackage{enumitem}

\usepackage{algorithm}
\usepackage[noend]{algpseudocode}
\usepackage{authblk}
\usepackage{eucal}
\usepackage{url}
\usepackage{csvsimple}
\usepackage{sectsty}
\usepackage{verbatim}
\usepackage{longtable}

\sectionfont{\fontsize{12}{15}\selectfont}

\newtheorem{thm}{Theorem}[section]
\newtheorem{lem}[thm]{Lemma}

\newtheorem{cor}[thm]{Corollary}
\newtheorem{conj}[thm]{Conjecture}
\newtheorem{example}{Example}
\newtheorem{defn}[thm]{Definition}
\newtheorem{rem}[thm]{Remark}

\newtheorem{prop}[thm]{Proposition}

\begin{document}
\title{The Maximum Length of Circuit Codes With Long Bit Runs and a New Characterization Theorem}
\author{Kevin M. Byrnes
\thanks{E-mail:\texttt{dr.kevin.byrnes@gmail.com}}}
\maketitle

\begin{abstract}
We study circuit codes with long bit runs (sequences of distinct transitions) and derive a formula for the maximum length for an infinite class of symmetric circuit codes with long bit runs.  This formula also results in an improved lower bound on the maximum length for an infinite class of circuit codes without restrictions on symmetry or bit run length.  We also present a new characterization of circuit codes of spread $k$ based on a theorem of Deimer.  
\end{abstract}

\footnotesize
\textbf{Keywords: }Circuit Code, Snake in the Box, Coil in the Box, k-Coil, Error Correcting Code
\normalsize

\section{Introduction}
Let $I(d)$ denote the graph of the $d$-dimensional hypercube.  A simple cycle $C=(x_1,\ldots,x_N)$ in $I(d)$ is called a \emph{circuit}.  A circuit $C$ has three important characteristics: its ambient dimension $d$, its \emph{spread} $k$ which is the minimum distance in $I(d)$ two vertices $x_i$ and $x_j\in C$ can have if they are not ``close'' in $C$, and its length $N$.  Let $G$ be a subgraph of $I(d)$ and let $x$ and $y$ be vertices of $G$.  Define $d_G(x,y)$ as the minimum number of edges that need to be traversed in $G$ to travel from $x$ to $y$ (with $d_G(x,y) = \infty$ if no such path exists).  A circuit code of spread $k$ can then be formally defined as follows.

\begin{defn}
\label{def1}
A subgraph $C$ of $I(d)$ is a circuit code of spread $k$ (a $(d,k)$ circuit code)~if:
\begin{enumerate}
\setlength{\itemsep}{0pt}
\setlength{\parskip}{0pt}
\item $C$ is a circuit.
\item If $x$ and $y$ are vertices of $C$ with $d_{I(d)}(x,y)<k$ then $d_C(x,y)=d_{I(d)}(x,y)$.
\end{enumerate}
\end{defn}

A useful alternate characterization of a spread $k$ circuit code was given by Klee.
\begin{lem}[Klee \cite{Klee} Lemma 2]
\label{KleeLemma2}
A $d$-dimensional circuit code $C$ of length $N\ge 2k$ has spread $k$ if and only if for all vertices $x,y\in C$, ${d_C(x,y)\ge k \Rightarrow d_{I(d)}(x,y)\ge k}$.
\end{lem}

Circuit codes were first introduced in \cite{Kautz}.  Since then they have been extensively studied as both combinatorial objects (generalizing the well-known Snake in the Box problem \cite{KleeMAA}) and as a type of error-correcting code (\cite{Paterson, Hood, Casella, Byrnes, Zinovik,Chebiryak}).\footnote{For a survey of circuit codes, see \cite{Grunbaum} chapter 17.}

In this note we study circuit codes with long bit runs (sequences of distinct transitions between cyclically consecutive vertices).  We develop structural results and upper and lower bounds on the maximum length of such codes when they are symmetric (Lemmas \ref{XiUpperBd} - \ref{lem4k2l}).  These results lead to an exact formula for the maximum length of a symmetric $(d,k,r)$ circuit code when the spread $k$ is odd and the bit run length $r$ is maximum relative to $d$ and $k$ (Theorem \ref{thm4k2l}).  Furthermore, Theorem \ref{thm4k2l} results in an improved lower bound on the maximum length of a $(d,k)$ circuit code, without restriction on symmetry or bit run length, when $d$ and $k$ satisfy the same assumptions as in Theorem \ref{thm4k2l} (Corollary \ref{cor4k2l}).  We also prove a new characterization theorem for circuit codes with spread $k$ based upon a necessary condition of \cite{Deimer} (Theorem \ref{CharThm}).

\section{Transition Sequences}
\label{TransitionSequences}
Each vertex of $I(d)$ corresponds to a binary vector of length $d$, so for every circuit $C=(x_1,\ldots,x_N)$ of $I(d)$ we can define a \emph{transition sequence} $T=(\tau_1,\ldots,\tau_N)$ where $\tau_i$ denotes the position in which $x_i$ and $x_{i+1}$ (or $x_N$ and $x_1$) differ.  Using the convention that $x_1=\vec{0}$ for any circuit, we see that the transition sequence corresponds uniquely to the edges in $C$.  Since $I(d)$ is bipartite this implies $|T|$ is even \cite{Harary}.

Define a \emph{segment} of a sequence $T=(\tau_1,\ldots,\tau_N)$ as a subsequence of cyclically consecutive elements.  For any $x_i,x_j\in C=(x_1,\ldots,x_N)$ with $i<j$ there are exactly two segments in $T$ between $x_i$ and $x_j$, corresponding to the two paths in $C$ traversing the edges: $x_ix_{i+1},\ldots,x_{j-1}x_j$ and $x_jx_{j+1},\ldots,x_{N-1}x_N,x_Nx_1,\ldots,x_{i-1}x_i$.  These segments are 
$(\tau_i,\tau_{i+1},\ldots,\tau_{j-1})$ and $(\tau_j,\tau_{j+1},\ldots,\tau_N,\tau_1,\ldots,\tau_{i-1})$.  If $i=j$ then the two segments are $\varnothing$ and $T$.  These segments are called complements because they partition $T$.  If $\hat{T}$ is a segment in $T$, its complement is denoted $\hat{T}^\complement$, and $(\hat{T}^\complement)^\complement=\hat{T}$.

The set of \emph{transition elements} $\{t_1,\ldots,t_m\}$ $(m\le d)$ of $T$ are the unique elements of $T$.  When $T$ is the transition sequence of a circuit each $t_i\in \{ t_1,\ldots,t_m\}$ must appear in $T$ an even number of times.  Without loss of generality, we assume $m=d$, otherwise the code can be embedded in a lower dimension. 

\begin{defn}
For a segment $\hat{T}$ of $T$, let $\delta(\hat{T})$ denote the number of transition elements in $\hat{T}$ that appear with odd parity.
\end{defn}

Observe that if $\hat{T}$ corresponds to a path in $C$ between vertices $x,y\in C$, then $d_{I(d)}(x,y)=\delta(\hat{T})$.
For notational convenience (especially in Lemmas \ref{XiUpperBd}-\ref{SymmCodeUb2}) we sometimes assume that a certain segment $\hat{T}$ of $T$ constitutes its first $m$ transitions or that the first $m$ transitions of $T$ are $(1,\ldots,m)$ for some appropriate value of $m$.  This is done without loss of generality since we can cyclically shift the indices of vertices in $C$ (and thus of the transitions in $T$) and are free to label $\{t_1,\ldots,t_d\}$ according to any permutation of $\{1,\ldots,d\}$.

\section{Circuit Codes with Long Bit Runs}
Given a $(d,k)$ circuit code $C$ with transition sequence $T$, the \emph{maximum bit run}, $\phi(C)$, denotes the length of the longest segment of $T$ that does not repeat a transition element.  Similarly, the \emph{minimum bit run}, $\xi(C)$, denotes the maximum length an arbitrary segment of $T$ can have without repeating a transition.  It is easy to see that $\phi(C)\le d$, and if $N=|C|>2k$ then $\xi(C)\ge k+1$.  The special case where $k=1$ and $C$ is a Hamiltonian circuit has been extensively studied \cite{GoddynLawrence, Goddyn, Lukito}, and there exist such $C$ where $\xi(C)\ge d-3\log_2 d$ \cite{Goddyn}.  An upper bound on $\phi(C)$ (in terms of $d$) is given by the following theorem.

\begin{thm}[Singleton \cite{Singleton} Theorem 3]
\label{SingletonDimensionThm}
If $C$ is a $(d,k)$ circuit code with $N=|C|>2d$, then $\phi(C)\ge k+2$ implies $d\ge k+1+\lfloor \frac{\phi(C)}{2}\rfloor$.
\end{thm}

Let $\mathcal{F}(d,k,r)$ denote the set of spread $k$ circuit codes in dimension $d$ with $\phi(C)\ge r$.  Note that $\mathcal{F}(d,k,r+1)\subseteq \mathcal{F}(d,k,r)$ and $\mathcal{F}(d,k+1,r)\subseteq \mathcal{F}(d,k,r)$.  Let $L(d,k,r)$ denote the maximum length of an element of $\mathcal{F}(d,k,r)$, i.e. $K(d,k)$ for $C\in \mathcal{F}(d,k,r)$.  Douglas \cite{Douglas2} proved the following results.

\begin{prop}[Douglas \cite{Douglas2} Remark (6)]
\label{DouglasRem5}
For $k$ even and $l$ odd, with $k\ge 2l-2$, $L(\frac{3k}{2}+\frac{l+1}{2},k,k+l)\le 4k+3l-1$.
\end{prop}

\begin{prop}[Douglas \cite{Douglas2} Remark (7)]
\label{DouglasRem6}
For $k$ odd and $l$ even, with $k\ge 2l+1$, $L(\frac{3k}{2}+\frac{l+1}{2},k,k+l)\le 4k+3l+2$.
\end{prop}

In addition \cite{Douglas2} established exact values for $K(d,k)$ for the following cases (which can be interpreted as tight upper bounds for the length of $C\in \mathcal{F}(d,k,k+1)$).  Notice that the upper bounds of Propositions \ref{DouglasRem5} and \ref{DouglasRem6} are not tight for the cases below.

\begin{thm}
\label{DouglasLengthThm}
\begin{enumerate}
\setlength{\itemsep}{0pt}
\setlength{\parskip}{0pt}
\item []
\item[(i)] $K(\frac{3k}{2} +2,k) = 4k+6$ for $k$ even.  (\cite{Douglas2} Theorem 3).
\item[(ii)] $K(\lfloor \frac{3k}{2} \rfloor +2,k) = 4k+4$ for $k$ odd.  (\cite{Douglas2} Theorem 4).
\item[(iii)] $K(\lfloor \frac{3k}{2} \rfloor +3,k) = 4k+8$ for $k$ odd and $\ge 9$.  (\cite{Douglas2} Theorem 5).
\end{enumerate}
\end{thm}

Formulas for the exact value of the maximum length of a circuit code are extremely rare, in fact the only non-trivial formulas known are those in Theorem \ref{DouglasLengthThm}.  The main result of this section (Theorem \ref{thm4k2l}) is a new formula for the maximum length of symmetric $C\in \mathcal{F}(\frac{3k}{2}+\frac{l+1}{2},k,k+l)$ when $k$ is odd and $l$ is even $\ge 2$ with $k\ge 2l+1$, and a new lower bound on $K(\frac{3k}{2}+\frac{l+1}{2},k)$ that improves upon the best known lower bound when $k$ and $l$ satisfy these conditions (Corollary \ref{cor4k2l}).

We begin with a technical lemma showing that codes
 with the longest bit run possible in their dimension (per Theorem \ref{SingletonDimensionThm}) have $\xi(C)$ minimum.  Our argument is a generalization of an approach used by \cite{Singleton}.

\begin{lem}
\label{XiUpperBd}
Let $C$ be a $(\frac{3k}{2}+\frac{l+1}{2},k,k+l)$ circuit code where:  $k$ is even and $l$ odd, or $k$ odd and $l$ even, and having length $N>2d=3k+(l+1)$.  Then $\xi(C)=k+1$.  Furthermore, any segment $\hat{T}=(\tau_{r+1},\ldots,\tau_{r+2k+l+1})$ of the transition sequence $T$ of $C$ where $\tau_{r+1},\ldots,\tau_{r+k+l}$ are all distinct has $\delta(\hat{T})=k$ and all $\frac{3k}{2}+\frac{l+1}{2}$ transition elements appear in $\hat{T}$.
\end{lem}

\begin{proof}
We assume that the transition sequence $T$ of $C$ begins with the segment:
\[
\hat{T} = \underbrace{\tau_1,\ldots,\tau_{k+l}}_{\omega_1},\underbrace{\tau_{k+l+1},\ldots,\tau_{2k+l+1}}_{\omega_2},\alpha
\]
where $\omega_1$ consists of $k+l$ distinct transitions and all the transitions in $\omega_2$ are distinct.  Note that since $C\in \mathcal{F}(\frac{3k}{2}+\frac{l+1}{2},k,k+l)$ its transition sequence $T$ contains a segment $\tilde{T}$ of $k+l$ distinct transitions, and the next $k+1$ transitions must all be distinct since $\xi(C)\ge k+1$.  Without loss of generality these $2k+l+1$ transitions are the first $2k+l+1$ transitions of $T$.

Let $n_1$ denote the number of transitions appearing once in $(\omega_1,\omega_2)$ and let $n_2$ denote the number of transitions appearing twice.  Then $n_1\ge k$ (as $k<|(\omega_1,\omega_2)|<N-k$, so $\delta((\omega_1,\omega_2))\ge k$), $n_1+n_2\le d$, and $n_1+2n_2=2k+l+1$.  Thus $2d\ge 3k+l+1$ with equality implying $n_1=k$.  Since $d=\frac{1}{2}(3k+(l+1))$ we see that $\delta((\omega_1,\omega_2))=n_1=k$ and $n_2=\frac{1}{2}(k+(l+1))=d-n_1$.  Hence all transitions occur in $(\omega_1,\omega_2)$ and thus in $\hat{T}$.

Now $\delta((\omega_1,\omega_2))=n_1=k$, so $\delta(\hat{T})=\delta((\omega_1,\omega_2,\alpha))=\delta((\omega_1,\omega_2))\pm 1$ must equal $k+1$ (as $N>2d)$.  Thus $\alpha$ occurs twice in $(\omega_1,\omega_2)$, so $\alpha \in \omega_2$.  Hence $(\omega_2,\alpha)$ is a segment of $T$ of size $k+2$ with a repeated transition, so $\xi(C)=k+1$.
\end{proof}

Recall that a \emph{symmetric} code is one whose transition sequence $T=(\tau_1,\ldots,\tau_N)$ satisfies $\tau_i=\tau_{i+N/2}$ for $1\le i\le N/2$.  
In the next two results we use the $(\omega_1,\omega_2)$ structure from Lemma \ref{XiUpperBd} and related observations to derive new upper bounds on the length of symmetric $C\in \mathcal{F}(\frac{3k}{2}+\frac{l+1}{2},k,k+l)$, substantially improving upon the general case upper bounds of Propositions \ref{DouglasRem5} and \ref{DouglasRem6}.

\begin{lem}
\label{SymmCodeUb1}
Let $k$ be even and $l$ be odd $\ge 3$ with $k\ge 2l-2$, or $k$ is odd and $l$ even $\ge 2$ with $k\ge 2l+1$.  Let $C\in \mathcal{F}(\frac{3k}{2}+\frac{l+1}{2},k,k+l)$ be symmetric, then $|C|\le 4k+2(l+1)$.
\end{lem}

\begin{proof}
Assume not, then without loss of generality the transition sequence $T$ of $C$ contains a segment
\[
\hat{T}=\underbrace{\tau_1,\ldots,\tau_{k+l}}_{\omega_1},\underbrace{\tau_{k+l+1},\ldots,\tau_{2k+l+1}}_{\omega_2},\alpha
\]
where all $k+l$ transitions in $\omega_1$ are distinct and all of the $k+1$ transitions in $\omega_2$ are distinct.  We know (by Lemma \ref{XiUpperBd}) that $\delta((\omega_1,\omega_2))=k$, thus $\alpha$ must appear an even number of times in $(\omega_1,\omega_2)$.  Since all transitions appear in $(\omega_1,\omega_2)$ it follows that $\alpha \in \omega_2$, so $\alpha=\tau_{k+l+1}$, and $\alpha \in \omega_1$, so $\tau_{k+l+1}\in \{\tau_1,\ldots,\tau_l\}$ (as $\xi(C)=k+1$).

Now consider the segment $\beta=(\alpha,\ldots,\tau_1,\ldots,\tau_l)$ of $T$.  If $k$ is even then $|\beta|\le \frac{l}{2}-\frac{3}{2}+l=\frac{3(l-1)}{2}$ by Proposition \ref{DouglasRem5}.  If $k$ is odd, then $|\beta|\le \frac{3l}{2}$ by Proposition \ref{DouglasRem6}.  In the even case, $\frac{3(l-1)}{2}<k$ since $k\ge 2l-2$, in the odd case $\frac{3l}{2}<k$ since $k\ge 2l+1$.  Hence by Definition \ref{def1} we require $\delta(\beta)=|\beta|$ in both cases, i.e. all transitions in $\beta$ must be distinct.  But since $\alpha \in \{\tau_1,\ldots,\tau_{l}\}$ this is impossible and we reach a contradiction.
\end{proof}

\begin{lem}
\label{SymmCodeUb2}
Let $k$ be odd and $l$ even $\ge 2$ with $k\ge 2l+1$.  Then for $C\in \mathcal{F}(\frac{3k}{2}+\frac{l+1}{2},k,k+l)$ and symmetric, $|C|\le 4k+2l$.
\end{lem}

\begin{proof}
From Lemma \ref{SymmCodeUb1} we have $|C|\le 4k+2(l+1)$.  Assume for contradiction that $|C|=4k+2(l+1)$, then without loss of generality the transition sequence $T$ of $C$ has the form:
\small
\[
T=\underbrace{1,2,\ldots,l}_{\omega_1},\underbrace{l+1,\ldots,k+l}_{\omega_2},\underbrace{\alpha_1,\ldots,\alpha_k}_{\omega_3},\alpha_{k+1},
1,2,\ldots,l,l+1,\ldots,k+l,\alpha_1,\ldots,\alpha_k,\alpha_{k+1}.
\]
\normalsize
Let $\hat{T}$ denote the segment $(\omega_1,\omega_2,\omega_3,\alpha_{k+1})$.
From the Lemma \ref{XiUpperBd} we know that all $d=\frac{3k}{2}+\frac{l+1}{2}$ transition elements are used in $\hat{T}$, that each transition appears once or twice in $\hat{T}$, and that $\delta(\hat{T})=k$.  From this we have $k\le \delta((\omega_1,\omega_2,\omega_3))=\delta(\hat{T})\pm 1$, so $\delta((\omega_1,\omega_2,\omega_3))=k+1$ and $\alpha_{k+1}\in (\omega_1,\omega_2,\omega_3)$.  Hence all transition elements appear in $(\omega_1,\omega_2,\omega_3)$.  Furthermore, since $\xi(C)=k+1$, $\alpha_{k+1}\not\in \omega_3$, so $\alpha_{k+1}\in (\omega_1,\omega_2)$.

Define $\beta$ as $(\omega_2,\omega_3)$.  Clearly $|\beta|=2k$, and since $|C|=4k+2(l+1)$ this means $\delta(\beta)\ge k$.  A precise value for $\delta(\beta)$ is given by

\begin{equation}
\label{eqn1}
\delta(\beta)=(d-(k+l))+s_1+(k-s_2)
\end{equation}

where $d-(k+l)$ is the number of transition elements in $\{1,\ldots,d\}$ that have not appeared in $(\omega_1,\omega_2)$, $s_1$ is the number of transitions in $\omega_3$ that appear in $\omega_1$, and $s_2$ is the number of transitions in $\omega_3$ that appear in $\omega_2$.  Furthermore

\begin{equation}
\label{eqn2}
|\omega_3|=k=(d-(k+l))+s_1+s_2.
\end{equation}

From $(\ref{eqn1})$ and $(\ref{eqn2})$ we deduce the following relationships:
\begin{equation}
\label{eqn3}
(1/2)(k+l-1)=s_1+s_2
\end{equation}

\begin{equation}
\label{eqn4}
s_1\ge \lceil (l-1)/2 \rceil = l/2
\end{equation}

\begin{equation}
\label{eqn5}
s_2 \le \lfloor (1/2)k\rfloor = (1/2)(k-1).
\end{equation}

Now consider the segment $\gamma=(\omega_3,\alpha_{k+1},\omega_1)$ of $T$.  This segment has size $k+1+l$ and so $k\le \delta(\gamma)$.  An upper bound on $\delta(\gamma)$ is given by $(d-(k+l))+(l-s_1)+s_2+1$.  Thus we require
\begin{equation}
\label{eqn6}
2k-d+s_1\le s_2+1.
\end{equation}

The only values of $s_1$ and $s_2$ consistent with $(\ref{eqn3})-(\ref{eqn6})$ are $s_1=\frac{l}{2}$ and $s_2=\frac{k-1}{2}$.  Plugging these into $(\ref{eqn1})$ we get $\delta(\beta)=k+1$.  Recall that we have shown $\alpha_{k+1}\in (\omega_1,\omega_2)$.  Clearly $\alpha_{k+1}\not \in \omega_1$ since $(\alpha_{k+1},\omega_1)$ is a segment of $T$ of size $<k$ and thus all transitions in $(\alpha_{k+1},\omega_1)$ must be distinct since $C$ has spread $k$.

Therefore $\alpha_{k+1}\in \omega_2$ and so $\delta(\beta,\alpha_{k+1})=\delta(\beta)-1=k$.  Thus $\delta(\beta,\alpha_{k+1},1)$ must equal $k+1$ and hence $1\not \in \{\alpha_1,\ldots,\alpha_{k+1}\}$.  But since $\delta(\hat{T})=k$ and $1$ occurs only once in $\hat{T}$, this implies $\delta((2,\ldots,l,\omega_2,\alpha_1,\ldots,\alpha_{k+1}))=k-1$ a contradiction.  

Therefore we reach a contradiction and conclude that $|C|\le 4k+2l$.
\end{proof}

Unlike the case with the upper bounds of Propositions \ref{DouglasRem5} and \ref{DouglasRem6}, the upper bound of $4k+2l$ can be achieved.  This is shown in the following example which is based on a canonical augmentation approach similar to \cite{Ostergard}.

\begin{example}
\label{Ex1}
We wish to see if a symmetric $(16,9,9+4)$ circuit code of length $44$ exists.  If such a code $C$ exists then, without loss of generality, it has transition sequence
\[
T=\underbrace{1,\ldots,4}_{\omega_1},\underbrace{5,\ldots,13}_{\omega_2},\underbrace{\alpha_1,\ldots,\alpha_9}_{\omega_3},1\ldots,4,5\ldots,13,\alpha_1,\ldots,\alpha_9.
\]

Since $T$ is symmetric, all $16$ transition elements appear in $(\omega_1,\omega_2,\omega_3)$.  By construction and the fact that $\xi(C)=10$, each transition element appears once or twice in $(\omega_1,\omega_2,\omega_3)$.
Using $\delta((\omega_2,\omega_3))\ge 9$ and $\delta((\omega_3,1,\ldots,4))\ge 9$ and the proof approach of Lemma \ref{XiUpperBd} we find that: $2$ members of $\omega_3$ are in $\omega_1$, $4$ members of $\omega_3$ are in $\omega_2$, and $3$ members of $\omega_3$ are in $\{14,15,16\}$.  From the structure of $T$ we also deduce that $\alpha_i \not\in \{1,\ldots,i\}$ for $1\le i\le 9$.  Hence $2$ members of $\{2,3,4\}$ are in $\{\alpha_1,\alpha_2,\alpha_3\}$.  This greatly reduces the search space for possible $(16,9,9+4)$ codes.

Using canonical augmentation of $(\omega_1,\omega_2)$ plus the refinements mentioned above, we find the following $(16,9,9+4)$ code of length $44$:
\small
\[
T=1,2,3,4,5,6,7,8,9,10,11,12,13,2,4,6,14,8,15,10,16,12,
\]
\[
1,2,3,4,5,6,7,8,9,10,11,12,13,2,4,6,14,8,15,10,16,12.
\]
\normalsize
\end{example}

Example \ref{Ex1} suggests a general structure for symmetric $(\frac{3k}{2}+\frac{l+1}{2},k,k+l)$ codes with $k$ odd and $l$ even $\ge 2$ having length $4k+2l$.  Define the code $C$ by the symmetric transition sequence whose first half is as follows:
\begin{equation}
\label{Form4k2l}
\underbrace{1,\ldots,k+l}_{\omega_1 \text{ length }=k+l},\underbrace{2,4,\ldots,2l-2}_{\omega_2 \text{ length }=l-1},\underbrace{\gamma_1,\beta_{1},\gamma_2,\beta_{2},\ldots,\gamma_{d-(k+l)},\beta_{d-(k+l)}}_{\omega_3 \text{ length }=k+1-l}.
\end{equation}

Here $\beta_1$ through $\beta_{d-(k+l)}$ are $2l,2l+2,\ldots,k+l-1$ and $\gamma_1$ through $\gamma_{d-(k+l)}$ are $k+l+1,k+l+2,\ldots,d$.

\begin{lem}
\label{lem4k2l}
The symmetric transition sequence in (\ref{Form4k2l}) defines a circuit code $C\in \mathcal{F}(\frac{3k}{2}+\frac{l+1}{2},k,k+l)$ for $k$ odd and $l$ even $\ge 2$ having length $4k+2l$.
\end{lem}
\begin{proof}
Let $T$ be the symmetric transition sequence whose first half is defined as in (\ref{Form4k2l}) and let $C$ be the attendant $d (=\frac{3k}{2}+\frac{l+1}{2})$-dimensional circuit code defined by $T$.  Evidently $\phi(C)\ge k+l$ and $|C|=4k+2l$, so all that needs to be verified is the spread.

Partition $T$ as $(\underbrace{\omega_1,\omega_2,\omega_3}_{A},\underbrace{\omega_1,\omega_2,\omega_3}_{B})$ where $\omega_1=1,\ldots,k+l$, $\omega_2=2,4,\ldots,2l-2$, and $\omega_3=\gamma_1,\ldots,\beta_{d-(k+l)}$.  Select $x$ and $y$ arbitrarily from all pairs of vertices $u,v\in C$ with $d_C(u,v)\ge k$.  We will prove that $d_{I(d)}(x,y)\ge k$, hence by Lemma \ref{KleeLemma2} $C$ has spread $\ge k$.  Let $\hat{T}=(\tau_i,\ldots,\tau_{j-1})$ denote a shortest segment in $T$ between $x$ and $y$.

\textbf{Case $I$.}  $\hat{T}$ is a segment of $A$.\\
Hence $\tau_i\in \omega_m$ and $\tau_{j-1}\in \omega_n$ with $m\le n$ (else $|\hat{T}|>2k+l$).

\textbf{Subcase $I.1$.}  $\tau_i$ and $\tau_{j-1}$ are both in $\omega_s$ for $s\in \{1,2,3\}$.\\
Each $\omega_s$ contains no repeated transition elements, hence $\delta(\hat{T})=|\hat{T}|=d_C(x,y)\ge k$.

\textbf{Subcase $I.2$} $\tau_i\in \omega_1 \text{ and }\tau_{j-1}\in \omega_2$.\\
Then $\tau_i=a\in \{1,\ldots,k+l\}$ and $\tau_{j-1}=2b\in \{2,4,\ldots,2l-2\}$.  We may assume $2b\ge a$ or else no transitions occur more than once (and hence $\delta(\hat{T})=|\hat{T}|\ge k$).  Only those transitions in $D=\{2\lceil \frac{a}{2}\rceil,2\lceil \frac{a}{2}\rceil+2,\ldots,2b\}$ are repeated and thus $\delta(\hat{T})=(k+l-a+1)+b-2|D|$.  The size of $D$ is $b-\lceil \frac{a}{2}\rceil +1$, so $\delta(\hat{T})=(k+l-a+1)+b-2(b-\lceil \frac{a}{2}\rceil +1)\ge k+l-b-1\ge k$.

\textbf{Subcase $I.3$.} $\tau_i \in \omega_1 \text{ and }\tau_{j-1}\in \omega_3$.\\
Then $\tau_i=a\in \omega_1$ and $\tau_{j-1}=\beta_s$ or $\gamma_s\in \omega_3$.  Consider $\tilde{T}=(\tau_i=a,\ldots,k+l,2,4,\ldots,\beta_u)$ where $u=s$ if $\tau_{j-1}=\beta_s$ and $u=s-1$ if $\tau_{j-1}=\gamma_s$ (with $\beta_0=2l-2$).  Since each $\gamma_u$ occurs only once in $A$ it is clear that $\delta(\hat{T})\ge \delta(\tilde{T})$.  We may assume $2\lceil \frac{a}{2}\rceil \le \beta_u$ since otherwise no transitions are repeated in $\tilde{T}$ and hence in $\hat{T}$, implying $\delta(\hat{T})\ge k$.  Only the transitions in $D=\{2\lceil \frac{a}{2}\rceil, 2\lceil \frac{a}{2}\rceil +2,\ldots, \beta_u=2(l+u-1)\}$ are repeated.  Since $|D|=(l+u-1 -\lceil \frac{a}{2}\rceil)+1$ we have $\delta(\tilde{T})=(k+l-a+1)+(l-1)+2u-2|D|\ge k$.

\textbf{Subcase $I.4$.}  $\tau_i\in \omega_2 \text{ and } \tau_{j-1}\in \omega_3$.\\
Since all transitions in $(\omega_2,\omega_3)$ are distinct, $\delta(\hat{T})=|\hat{T}|=d_C(x,y)\ge k$.

Clearly the analysis is the same if both $\tau_i$ and $\tau_{j-1} \in B$.\\

\textbf{Case $II$.}  $\tau_i\in A \text{ and } \tau_{j-1}\in B$.\\
In this case we have $\tau_i\in \omega_m$ and $\tau_{j-1}\in \omega_n$ where $m\ge n$ (else $|\hat{T}|>2k+l$ and $\hat{T}$ is not a shortest segment in $T$ between $x$ and $y$.)

\textbf{Subcase $II.1$.}  $\tau_i \text { and } \tau_{j-1}$ are both in $\omega_s$ for $s\in \{1,2,3\}$.\\
Then $\tau_i$ is the $a$th element of $\omega_s$ in $A$ and $\tau_{j-1}$ is the $b$th element of $\omega_s$ in $B$, and since $\hat{T}$ is a shortest segment between $x$ and $y$ in $T$ we have $b<a$.  Since $\delta(\hat{T})$ does not depend on the ordering of the transitions in $\hat{T}$, rearrange $\hat{T}$ as: $\hat{T}=\underbrace{1,\ldots,\tau_{j-1}}_{\text{was in } B},\underbrace{\tau_i,\ldots,\beta_{d-(k+l)}}_{\text{was in } A}.$
Define $m$ as $j-(2k+l)$, since $C$ is symmetric the sequence of values of the transitions in $\hat{T}$ is the same as in $T'=(\tau_1,\tau_2,\ldots,\tau_{m-1},\tau_i,\tau_{i+1},\ldots,\tau_{2k+l})$ (note since $b<a$ we have $m-1<i$) and thus $\delta(\hat{T})=\delta(T')$.  Furthermore $T'$ is a subsequence of ${\tilde{T}=(\tau_1=1,\ldots,\tau_{2k+l}=\beta_{d-(k+l)})}$ so the only transitions occurring twice are those even elements in $\{1,\ldots,k+l\}$ and no transition occurs three times or more.  Observe that by construction of $T$, $\delta(\tilde{T})=k+1$.  Also note that the only transitions in $\tilde{T}$ absent from $T'$ are $(\tau_m,\tau_{m+1},\ldots,\tau_{i-1})$.

Suppose $s=1$, then $(\tau_m,\tau_{m+1},\ldots,\tau_{i-1})=(\tau_{m-1}+1,\tau_{m-1}+2,\ldots,\tau_i -1)$.  Let $D=\{\tau_{m-1}+1,\ldots,\tau_i -1\}$, then $\delta(\hat{T})=\delta(\tilde{T})-|\{t\in D \ | t \text{ odd }\}| + |\{ t\in D \ | t \text{ even }\}| \ge k+1 - |\{ t \in D \ | t \text{ odd }\}| + (|\{ t \in D \ | t \text{ odd }\}| - 1) \ge k$.

Suppose $s=2$, then trivially $\delta(\hat{T})\ge \delta(\tilde{T})=k+1$ as we are only removing single instances of transitions that appear twice in $1,\ldots,\beta_{d-(k+l)}$.

Finally, suppose $s=3$.  Each transition in $D_1=\{\tau_m,\tau_{m+1},\ldots,\tau_{i-1}\}\cap \{1,\ldots,k+l\}$ increases $\delta(\hat{T})$ by $1$ relative to $\delta(\tilde{T})$ as we are removing from $\tilde{T}$ a single instance of a duplicated transition.  Each transition in $D_2=\{\tau_m,\tau_{m+1},\ldots,\tau_{i-1}\}-\{1,\ldots,k+l\}$ decreases $\delta(\hat{T})$ by $1$ relative to $\delta(\tilde{T})$, as we are removing from $\tilde{T}$ a transition that only occurred once.  Since the elements in $D_1$ and $D_2$ alternate in $(\tau_m,\tau_{m+1},\ldots,\tau_{i-1})$ (as the $\beta_u$'s and $\gamma_u$'s) we see that $\delta(\hat{T})\ge \delta(\tilde{T})-1=k$.

\textbf{Subcase $II.2$.}  $\tau_i\in \omega_3 \text{ and } \tau_{j-1}\in \omega_1$.\\
Then $\tau_i$ is the $a$th element of $\omega_3$ in $A$ and $\tau_{j-1}$ is the $b$th element of $\omega_1$ in $B$.  If $b<\beta_{\lceil \frac{a}{2}\rceil}=2(l+\lceil \frac{a}{2}\rceil-1)$ then all transitions in $\hat{T}$ are distinct, so $\delta(\hat{T})\ge k$.  Otherwise, only those transitions in $D=\{ \beta_{\lceil \frac{a}{2}\rceil}=2(l+\lceil \frac{a}{2}\rceil -1),2(l+\lceil \frac{a}{2}\rceil),\ldots,2\lfloor\frac{b}{2}\rfloor\}$ are repeated.  Since $|D|\le (b-2(l+\lceil \frac{a}{2}\rceil -1)/2 +1$ we have $\delta(\hat{T})\ge (k+1-l-a+1)+b-2|D|\ge k+l-2\ge k$.

\textbf{Subcase $II.3$.}  $\tau_i\in \omega_3 \text{ and }\tau_{j-1}\in \omega_2$.\\
Then $\tau_i$ is the $a$th element of $\omega_3$ and $\tau_{j-1}=2b\in \{2,4,\ldots,2l-2\}$.  Adding each element in $\{2b+2,2b+4,\ldots,2l-2\}$ to $\hat{T}$ to get a new segment $\tilde{T}=(\tau_i,\ldots,\beta_{d-(k+l)},\omega_1,\omega_2)$ results in $\delta(\hat{T})\ge \delta(\tilde{T})$ since each of the added transitions occurred exactly once in $\hat{T}$.  If $a=1$ then $\tilde{T}$ can be rearranged as $(1,\ldots,\beta_{d-(k+l)})$ and $\delta(\hat{T})\ge \delta(\tilde{T})>k$.  If $a>1$ and $|\tilde{T}|=2k+l-1$, update $\tilde{T}\rightarrow (\tilde{T},\gamma_1)$.  Now $\delta(\hat{T})\ge \delta(\tilde{T})-1$ and $\delta(\tilde{T})=\delta(1,\ldots,\beta_{d-(k+l)})$ (after rearrangement) $=k+1$.  Otherwise ($a>1$ and $|\tilde{T}|\le 2k+l-2$), update $\tilde{T} \rightarrow (\tilde{T},\gamma_1,\beta_1)$.  Then $\delta(\tilde{T})=\delta(\tau_i,\ldots,\beta_{d-(k+l)},\omega_1,\omega_2)$ $+1$ (from $\gamma_1$) $-1$ (from $\beta_1$) $\le \delta(\hat{T})$.  Now $\tilde{T}$ falls under subcase $II.1$ (with $s=3$), hence $k\le \delta(\tilde{T})\le \delta(\hat{T})$.  In both $a>1$ cases we have $\delta(\hat{T})\ge k$.

\textbf{Subcase $II.4$.}  $\tau_i\in \omega_2 \text{ and } \tau_{j-1}\in \omega_1$.\\
Then $\tau_i=2a\in \{2,4,\ldots,2l-2\}$ and $\tau_{j-1}$ is the $b$th element of $\omega_1$.  If $2a>b$ then no transitions are repeated in $\hat{T}$ and $\delta(\hat{T})=|\hat{T}|=d_C(x,y)\ge k$.  Otherwise ($2a\le b$) add $2,4,\ldots,2a-2$ to the beginning of $\hat{T}$ to get a new segment $\tilde{T}=(\omega_2,\omega_3,1,\ldots,b=\tau_{j-1})$ with $\delta(\hat{T})\ge \delta(\tilde{T})$.  If $b=k+l$ then $\tilde{T}$ can be rearranged as $(1,\ldots,\beta_{d-(k+l)})$ and $\delta(\hat{T})\ge \delta(\tilde{T})>k$.  If $b<k+l$ and $|\tilde{T}|=2k+l-1$, update $\tilde{T}\rightarrow (k+l,\tilde{T})$.  Now $\delta(\hat{T})\ge \delta(\tilde{T})-1$ and $\delta(\tilde{T})=\delta(1,\ldots,\beta_{d-(k+l)})=k+1$.  Otherwise ($b<k+l$ and $|\tilde{T}|\le 2k+l-2$), update $\tilde{T}\rightarrow (k+l-1,k+l,\tilde{T})$.  Then $\delta(\tilde{T})=\delta(\omega_2,\omega_3,1,\ldots,b)$ $+1$ (from $k+l$) $-1$ (from $k+l-1$) $\le \delta(\hat{T})$.  Now $\tilde{T}$ falls under subcase $II.1$ (with $s=1$), hence $\delta(\tilde{T})\ge k$.  In both $b<k+l$ cases we have $\delta(\hat{T})\ge k$.

Thus in all cases $d_{I(d)}(x,y)=\delta(\hat{T})\ge k$, proving the claim.
\end{proof}

Following from Lemmas \ref{SymmCodeUb2} and \ref{lem4k2l} we have a formula for the maximum length for a certain class of symmetric $(d,k,r)$ circuit codes.

\begin{thm}
\label{thm4k2l}
Let $k$ be odd and let $l$ be even $\ge 2$ with $k\ge 2l+1$.  Then the maximum length of a symmetric code $C\in \mathcal{F}(\frac{3k}{2}+\frac{l+1}{2},k,k+l)$ is exactly $4k+2l$.
\end{thm}

The best known lower bound on $K(d,k)$ for general $d$ and odd $k$ was given in \cite{Singleton} (stronger bounds are known for $k=2,3,4$ \cite{AK, Singleton, Byrnes} or for fixed $k$ as $d \rightarrow \infty$ \cite{PN}):

\begin{equation}
\label{singleton_eqn}
K(d,k) \ge (k+1)2^{\lfloor \frac{2d}{k+1}\rfloor -1}, \text{ when } k \text{ odd and }\left \lfloor \frac{2d}{k+1}\right \rfloor \ge 2.
\end{equation}

For $k$ and $l$ satisfying the conditions of Theorem \ref{thm4k2l} and $d=\frac{3k}{2}+\frac{l+1}{2}$ we have $\lfloor \frac{2d}{k+1}\rfloor =3$ and thus (\ref{singleton_eqn}) implies $K(d,k)\ge 4k+4$.  Clearly any symmetric $C\in \mathcal{F}(d,k,r)$ is a $(d,k)$ circuit code and $4k+4\le 4k+2l$ for $l\ge 2$.  Thus Theorem \ref{thm4k2l} implies an improved lower bound on $K(\frac{3k}{2}+\frac{l+1}{2},k)$.

\begin{cor}
\label{cor4k2l}
Let $k$ be odd and let $l$ be even $\ge 2$ with $k\ge 2l+1$.  Then $K(\frac{3k}{2}+\frac{l+1}{2},k)\ge 4k+2l$.
\end{cor}

\begin{rem}
\label{rem1}
In the proof of Theorem \ref{DouglasLengthThm} $(ii)$ in \cite{Douglas2} it is shown that all maximum length $(\lfloor \frac{3k}{2}\rfloor+2,k)$ codes ($k$ odd) are isomorphic.  This is not true for $K(\lfloor \frac{3k}{2}\rfloor +3,k)$ codes with $k$ odd.  Consider the $(16,9)$ code $C$ defined by the transition sequence $T$ of Example \ref{Ex1}, it has length $44=K(16,9)$ by Theorem \ref{DouglasLengthThm} $(iii)$.  Another $(16,9)$ code of length $44$ is $C'$ given by
\small
\[
T'=1,11,2,12,3,13,4,14,5,16,15,6,11,7,12,8,13,9,14,16,10,15,
\]
\[
1,11,2,12,3,13,4,14,5,16,15,6,11,7,12,8,13,9,14,16,10,15.
\]
\normalsize
Because $\phi(C)=13$ and $\phi(C')=12$ the two codes are not isomorphic.
\end{rem}

An interesting implication of Theorem \ref{thm4k2l} and Theorem \ref{DouglasLengthThm} $(ii)$ and $(iii)$ is that $L(\frac{3k}{2}+\frac{l+1}{2},k,k+l)=K(\frac{3k}{2}+\frac{l+1}{2},k)$ for odd $k \ge 9$ and $l = 2 \text{ or } 4$.  For such $k$ and $l$, the additional constraints that $\phi(C)\ge k+l$ or even that $C$ be symmetric do not affect the maximum code length.  Since (Remark \ref{rem1}) not all such maximum length codes are isomorphic or satisfy $\phi(C)\ge k+l$, this implication appears non-trivial and leads us to conjecture the following generalization of Theorem \ref{DouglasLengthThm}.

\begin{conj}
\label{conj4k2l}
Let $k$ be odd $\ge 9$ and $l$ even $\ge 2$ with $k\ge 2l+1$.  Then there exists symmetric $C\in \mathcal{F}(\frac{3k}{2}+\frac{l+1}{2},k,k+l)$ attaining length $K(d,k)$.  Hence $K(\frac{3k}{2}+\frac{l+1}{2},k)=4k+2l$ for such $(k,l)$ pairs.
\end{conj}

\section{A Characterization Theorem for Circuit Codes of Spread $k$}
Let $C$ be a $(d,k)$ circuit code of length $N$ with transition sequence $T=(\tau_1,\ldots,\tau_N)$ and transition elements $\{t_1,\ldots,t_d\}$.  For each $i\in \{1,\ldots,d\}$ let $T^i$ define the subsequence of $T$ resulting from removing $t_i$.

\begin{defn}
\label{SubcircuitCodeDef}
Given a $(d,k)$ circuit code $C$ with transition sequence $T$, the $i$th \emph{subcircuit code} $C^i$ is the walk in $I(d)$ induced by the sequence $T^i$ for $i\in \{1,\ldots,d\}$
\end{defn}

Although it may not be apparent from Definition \ref{SubcircuitCodeDef}, as long as $C^i$ is sufficiently long it is a $(d-1,k-1)$ circuit code, as shown by the following result.

\begin{thm}[Deimer \cite{Deimer} Theorem 1]
\label{DeimerThm}
Let $C$ be an $(d,k)$ circuit code of length $N$ with transition sequence $T=(\tau_1,\ldots,\tau_N)$ and transition elements $\{t_1,\ldots,t_d\}$.  Let $n_i$ denote the number of times $t_i$ occurs in $T$ for $i\in \{1,\ldots,d\}$.  If $|C^i|\ge 2(k-1)$ then $C^i$ is a $(d-1,k-1)$ circuit code of length $N-n_i$.
\end{thm}

If $k\ge 1$ and $|C|>4(k-1)$ (as will be the case in Theorem \ref{CharThm}) then the requirement on $|C^i|$ is easily satisfied.  To see this, note that when $k\ge 1$ that $C$ contains no repeated vertices and hence no transition element $t_i$ can appear twice consecutively.  Thus $|C^i|\ge N/2 > 2(k-1)$ for each $i\in \{1,\ldots,d\}$.

Theorem \ref{DeimerThm}, in conjunction with our results from Section $3$ yields a corollary which appears to fill a ``gap'' in the results of Theorem \ref{DouglasLengthThm}.

\begin{cor}
\label{cor4k2l_even}
Let $k$ be even and $l$ odd $\ge 3$ with $k\ge 2l-2$.  Then $K(\frac{3k}{2}+\frac{l+1}{2},k)\ge 4k+2l$.  In particular this implies $K(\frac{3k}{2}+3,k)\ge 4k+10$ for $k$ even $\ge 8$.
\end{cor}
\begin{proof}
Define $k'=k+1$ and $l'=l-1$, then $k'$ is odd and $l'$ is even $\ge 2$ with $k'\ge 2l'+1$.  Let $d'=\frac{3k'}{2}+\frac{l'+1}{2}$ and let $C\in \mathcal{F}(d',k',k'+l')$ be defined by the transition sequence $T$ given in (\ref{Form4k2l}).  Then $|C|=4k'+2l'=4k+2l+2$.  Observe that $t_{d'}$ occurs only twice in $T$, so by Theorem \ref{DeimerThm}, $C^{t_{d'}}$ is a $(d'-1,k'-1)=(\frac{3k}{2}+\frac{l+1}{2},k)$ circuit code of length $4k+2l$, thus $K(\frac{3k}{2}+\frac{l+1}{2},k)\ge 4k+2l$.  Taking $l=5$ yields $K(\frac{3k}{2}+3,k)\ge 4k+10$ for $k\ge 8$.
\end{proof}

The converse of Theorem \ref{DeimerThm} also holds, giving an alternate (to Lemma \ref{KleeLemma2}) characterization for circuit codes of spread $k$ as we will now show.

\begin{thm}
\label{CharThm}
Let $k\ge 2$, $d\ge k$, and let $C$ be a $d$-dimensional circuit code with length $N>4(k-1)$.  Then $C$ has spread $k$ if and only if $C^i$ is a $(d-1,k-1)$ circuit code for $i=1,\ldots,d$.
\end{thm}

\begin{proof}
Suppose $C$ has spread $k$, it immediately follows that $C^i$ is a $(d-1,k-1)$ circuit code for $i=1,\ldots,d$ by Theorem \ref{DeimerThm}.  Now suppose that $C^i$ is a $(d-1,k-1)$ circuit code for $i=1,\ldots,d$, and let $x,y\in C$ such that $d_C(x,y)\ge k$.  We will show that $d_{I(d)}(x,y)\ge k$, by Lemma \ref{KleeLemma2} it follows that $C$ has spread $k$.

We begin with some definitions.  Let $T$ be the transition sequence of $C$, and let $\hat{T}$ be the shortest segment of $T$ between $x$ and $y$ and let $\hat{T}^\complement$ denote its complement.  Also let $\{t_{\alpha(1)},\ldots,t_{\alpha(m)}\}\subseteq \{t_1,\ldots,t_d\}$ denote the transition elements appearing in $\hat{T}$.  For $i=1,\ldots,m$ let $T^{\alpha(i)}$ denote the subsequence of $T$ formed by deleting $t_{\alpha(i)}$, let $C^{\alpha(i)}$ be the attendant subcircuit code, and let $x^{\alpha(i)}$ and $y^{\alpha(i)}$ be the projections of $x$ and $y$ onto this $d-1$ dimensional space.

Next, we make a crucial observation: within any segment of $T$ of length $\le k$ the transition element $t_i \ (i\in 1,\ldots,d)$ can appear at most once, otherwise $C^j$ would violate Definition \ref{def1} (for spread $k-1$) for each $j \neq i$.

Since $x\neq y$ we may assume that $t_{\alpha(1)}$ occurs an odd number of times in $\hat{T}$.  Then $d_{I(d)}(x,y)\ge d_{I(d-1)}(x^{\alpha(1)},y^{\alpha(1)})+1 \ge \min \{d_{C^{\alpha(1)}}(x^{\alpha(1)},y^{\alpha(1)}),k-1\}+1$.  Since $|\hat{T}|\ge k$ (and $|\hat{T}^\complement|\ge|\hat{T}|$) and $t_{\alpha(1)}$ can occur at most once per $k$ cyclically consecutive elements of $T$ we have $d_{C^{\alpha(1)}}(x^{\alpha(1)},y^{\alpha(1)}) \ge k-1$ and hence $d_{I(d)}(x,y)\ge k$, proving the claim\footnote{In $C^{\alpha(1)}$ the direction of the shortest segment of $T^{\alpha(1)}$ between $x^{\alpha(1)}$ and $y^{\alpha(1)}$ (e.g. from $x^{\alpha(1)}$ to $y^{\alpha(1)}$ or from $y^{\alpha(1)}$ to $x^{\alpha(1)}$) may be reversed from the direction of $\hat{T}$.  But $N$ is sufficiently large so that $d_{C^{\alpha(1)}}(x^{\alpha(1)},y^{\alpha(1)})$ remains $\ge k-1$.}.
\end{proof}

We remark that Theorem \ref{CharThm} may be of practical interest since it suggests circuit codes have a decomposition structure that might be exploitable by parallelized algorithms, for example when verifying the spread.

\section{Conclusions}
In this note we presented several new results on circuit codes.  In Section $3$ we investigated circuit codes with long bit runs, establishing the exact value of $K(d,k)$ for symmetric $C\in \mathcal{F}(\frac{3k}{2}+\frac{l+1}{2},k,k+l)$ when $k$ is odd and $l$ is even $\ge 2$ with $k\ge 2l+1$ (Theorem \ref{thm4k2l}) and an improved lower bound on $K(\frac{3k}{2}+\frac{l+1}{2},k)$ for such $(k,l)$ pairs (Corollary \ref{cor4k2l}).  In Section $4$ we proved a new characterization of circuit codes of spread $k$ (Theorem \ref{CharThm}) that is a converse to Deimer's Theorem, and improved the lower bound on $K(\frac{3k}{2}+\frac{l+1}{2},k)$ when $k$ even and $l$ odd $\ge 3$ with $k\ge 2l-2$ (Corollary \ref{cor4k2l_even}).

Several interesting questions remain open for investigation.  Proving Conjecture \ref{conj4k2l} even for the case $l=6$ would represent notable progress in computing exact values for $K(d,k)$.  Furthermore, although the structural and upper and lower bounds presented here were developed for proving Theorem \ref{thm4k2l} it would be interesting to see if they could be adapted to other types of circuit codes (e.g. single track codes \cite{Hiltgen}) or different $(d,k)$ ranges.  Finally, it would be interesting to see if Theorem \ref{CharThm} could be further developed to lead to an efficient parallel algorithm.
\\
\\
\small
\textbf{Acknowledgements: }The author thanks Stephen Chestnut, Donniell Fishkind, and Florin Spinu for generously reviewing earlier versions of this paper and for many helpful suggestions.
\normalsize

\small
\bibliographystyle{plain}
\bibliography{CircuitCodesChar}
\normalsize

\end{document}